\documentclass[11pt, letterpaper]{amsart}
\usepackage[left=1in,right=1in,bottom=1in,top=1in]{geometry}
\usepackage{amsmath,amsthm,amssymb,mathtools}
\usepackage{calligra,mathrsfs}
\usepackage[all,cmtip]{xy}
\usepackage[noadjust]{cite}
\usepackage{hyperref}
\hypersetup{hidelinks}
\usepackage{enumerate}
\usepackage{comment}
\usepackage{graphicx}
\usepackage{enumitem} \setlist[enumerate]{label={\upshape(\roman*)}}

\usepackage[T2A,T1]{fontenc}
\usepackage[utf8]{inputenc}
\usepackage[russian,english]{babel}

\newcommand{\on}[1]{\operatorname{#1}}
\newcommand{\mathfont}{\mathbf}

\newcommand{\Z}{\mathfont Z}
\newcommand{\Q}{\mathfont Q}

\newcommand{\FF}{\mathfont F}

\DeclareFontFamily{OT1}{rsfs}{}
\DeclareFontShape{OT1}{rsfs}{n}{it}{<-> rsfs10}{}
\DeclareMathAlphabet{\mathscr}{OT1}{rsfs}{n}{it}

\newcommand{\Gscr}{\mathscr{G}}

\newcommand{\Gal}{\on{Gal}}

\newcommand{\ord}{\on{ord}}

\mathchardef\mhyphen="2D

\newcommand{\D}{\mathbf{D}}

\renewcommand{\O}{\mathcal{O}}
\newcommand{\E}{\mathcal{E}}
\newcommand{\F}{\mathcal{F}}

\theoremstyle{plain}
\newtheorem{lem}{Lemma}
\newtheorem{thm}[lem]{Theorem}

\newtheorem{cor}[lem]{Corollary}

\theoremstyle{definition}
\newtheorem{defn}[lem]{Definition}

\newtheorem{rem}[lem]{Remark}

\numberwithin{equation}{section}
\numberwithin{lem}{section}

\DeclareMathOperator{\HOM}{\mathscr{H}\text{\kern -3pt {\calligra\large om}}\,}

\renewcommand{\o}[1]{\overline{#1}}

\DeclareMathOperator{\NP}{NP}

\DeclareMathOperator{\crys}{crys}

\renewcommand*{\ll}{\mathrm{ll}}

\title{Unramified Geometric Iwasawa theory of function fields}
\date{\today}

\author[B. Cais]{Bryden Cais}
\address{Department of Mathematics, The University of Arizona,
Tucson, Arizona, USA}
\email{cais@arizona.edu}

\author[D. Lewis]{Daniel Lewis}
\address{Independent researcher, Manchester, UK}
\email{danlewis92@gmail.com}

\usepackage[usenames,dvipsnames]{color}

\makeatletter
\@namedef{subjclassname@2020}{%
  \textup{2020} Mathematics Subject Classification}
\makeatother

\begin{document} 

\CompileMatrices
\UseTips

\begin{abstract}
    For an algebraically closed field $k$ of characteristic $p>0$,
    we study unramified $\Z_p$-towers $\{X_n\}_{n\ge 0}$ of smooth,
    projective and connected curves over $k$.  
    Our main result provides an Iwasawa-theoretic
description of the Newton polygons of their Jacobians, establishing that the slopes of the new part at level $n$ are
 eventually affine functions of $p^{-n}$ with rational coefficients.    
 \end{abstract}
 
\keywords{Iwasawa theory, $p$-divisible groups, Newton polygons}
\subjclass[2020]{11R23, 
14L05,  	
11R58, 
14H30
}

\thanks{The first author was supported by NSF grant numbers DMS-1902005 and DMS-2302072.}

\maketitle

\section{Introduction}\label{intro}

Fix an algebraically closed field $k$ of characteristic $p>0$ and
let $X_0$ be a smooth, projective, and connected curve over $k$, with genus $g$ and $p$-rank $f>0$. 
By a theorem of Shafarevich 
 \cite[Theorem 1.9]{crew84},
the maximal pro-$p$ quotient of the \'etale fundamental group of $X_0$ is a free pro-$p$ group
on $f$ generators; in particular, this profinite group admits a continuous surjection onto $\Z_p$.
Fix any such surjection, and let $\{X_n\}_{n\ge 0}$ be the corresponding
$\Z_p$-{\em tower} 
of \'etale Galois covers $X_n\rightarrow X_0$ of curves
over $k$
with $\Gal(X_n/X_0)\simeq \Z/p^n\Z$, compatibly with change in $n$.
By the Riemann--Hurwitz and Deuring--Shafarevich formulae, $X_n$ has genus $g_n\coloneq p^n(g-1)+1$
and $p$-rank $f_n\coloneq p^n(f-1)+1$.

Let $W\coloneq W(k)$ be the ring of Witt vectors of $k$,  
and 
put $K_0\coloneq W[1/p]$.
In this paper, we will study the 
$F$-isocrystals 
$$D_n\coloneq K_0\otimes_W H^1_{\crys}(X_n/W)$$
with $F$ given by the crystalline $p$-power Frobenius map,
as $n\rightarrow \infty$.  By the Dieudonn\'e--Manin classification, 
$D_n$ is determined up to isomorphism as an $F$-isocrystal by its multiset of slopes; equivalently, its {\em Newton polygon}.
Moreover, one has a canonical decomposition $D_n = D_n^{\ord}\oplus D_n^{(0,1)}$
where $D_n^{\ord}$ has slopes 0 and 1, each with multiplicity $f_n$, and $D_n^{(0,1)}$ has all slopes in $(0,1)$.

To state our main result, let $\Lambda\coloneq W[\![T]\!]$ be the ring of power series over $W$ in one variable, 
write $\O_{\E}$ for the $p$-adic completion of $\Lambda[1/T]$ and set $\E\coloneq\O_{\E}[1/p]$;
this is a complete, discretely valued field with valuation ring $\O_{\E}$, uniformizer $p$, and residue field $k(\!(T)\!)$.  
For an $F$-isocrystal $D$, we write $\NP(D)$
for its corresponding Newton polygon. 
For a discretely valued field $\F$ with valuation
$v$ normalized by $v(\varpi)=1$ for a chosen $\varpi\in\F^\times$,
and $Q=\sum_{i=0}^r a_i Z^{i}\in \F[Z]$ with $a_r\neq 0$, 
we write $\NP_{\F}(Q)$ or $\NP_{\varpi}(Q)$ for the {\em Newton polygon}
of $Q$ with respect to $v$, {\em i.e.}~the lower convex hull of $\{(r-i,v(a_{i}))\ :\ 0\le i\le r,\ a_i\neq 0\}$.
A finite {\em direct sum} of Newton polygons is the Newton polygon
whose multiset of slopes is the disjoint union of the multisets of slopes of the summands.

\begin{thm}\label{thm:main-intro}
    Set $d\coloneq g-f$.  There exist $\delta\in W[T]$ with $\delta(\zeta-1)\neq 0$ for every $p$-power root of unity $\zeta\in \o{K}_0$,
    and a monic polynomial 
    $P$ of degree $2d$ with coefficients in $\Lambda[1/\delta]\cap \O_{\E} \subseteq \E$
    such that
    $$
        \NP(D_n^{(0,1)}) = \NP(D_{n-1}^{(0,1)}) \oplus \NP_p(P\big|_{T=\zeta-1})^{\oplus(p^n-p^{n-1})}
    $$
    for all $n \ge 1$, 
    where $\zeta$ is any choice of primitive $p^n$-th root of unity.
    Viewing $P$ as a polynomial with coefficients in $\E$,
    the ``generic'' Newton polygon $\NP_{\E}(P)$ has endpoints $(0,0)$ and $(2d,d)$,
    all slopes in $[0,1]$, and underlying
    multiset of slopes invariant under $s\mapsto 1-s$. Moreover, 
    $\NP_p(P\big|_{T=\zeta-1})$ lies on or above $\NP_{\E}(P)$
    with the same endpoints, and converges uniformly to $\NP_{\E}(P)$ as $n\rightarrow \infty$.
\end{thm}

\begin{rem}\label{rem:delta}
    Note that the specialization $P\big|_{T=\zeta-1}$ makes sense
    as a polynomial with coefficients in $K_n\coloneq K_0(\mu_{p^n})$ as the image of $\delta$
    under the map $\Lambda \rightarrow K_n$ sending $T$
    to $\zeta-1$ is {\em nonzero} for $n\ge 1$.  
    While neither $\delta$ nor $P$
    are canonically determined by our fixed $\Z_p$-tower,
    the Newton polygons $\NP_{\E}(P)$ and $\NP_p(P\big|_{T=\zeta-1})$ 
    depend only on the tower, and it is possible to formulate
    much of Theorem \ref{thm:main-intro} intrinsically (see Theorem \ref{thm:NPabstract} below).  
    However, the more explicit
    Corollary \ref{cor:main} below requires the full strength of Theorem \ref{thm:main-intro} as stated.
\end{rem}

The Newton polygon $\NP_p(P\big|_{T=\zeta-1})$
can be described explicitly when $n\gg0$ as follows.  Let 
$$
s_1\le s_2 \le \cdots \le s_{d} \le 1-s_{d}\le 1-s_{d-1}\le\ldots\le 1-s_{1}
$$
be the (not necessarily distinct) slopes of the generic Newton polygon $\NP_{\E}(P)$.

\begin{cor}\label{cor:main}
There exist $b_1,\ldots ,b_d \in \Q$ such that
the multiset of slopes of $\NP_p(P\big|_{T=\zeta-1})$
is 
$$
    \left\{s_i + \frac{b_i}{p^{n-1}(p-1)},1-s_{i} - \frac{b_{i}}{p^{n-1}(p-1)}\ \Big|\ i=1,\ldots,d\right\}
$$
for all $n\gg0$.
We have $\sum_{j=1}^i b_j \ge 0$ for $1\le i \le d$  and if $s_i=s_{i+1}$ for some $i<d$, then $b_i \le b_{i+1}$.
\end{cor}

\begin{rem}
    In the special case of a $\Z_p$-tower $\{X_n\}_{n\ge 0}$ over $k=\overline{\FF}_p$, 
    each curve $X_n$ is defined over some {\em finite} field $k_n$.  Writing $h_n\coloneq \log_p|k_n|$,
    the map $F^{h_n}$ is then a $W(k_n)$-{\em linear} endomorphism of $H^1_{\crys}(X_n/W(k_n))$,
    and we may form its characteristic polynomial
$$
        \mathcal{C}(X_n/k_n, s) \coloneq \det\left(s-F^{h_n}\big|{H^1_{\crys}(X_n/W(k_n))}\right). 
$$
    By \cite{KatzMessing}, this monic polynomial of degree $2g_n$ has integer coefficients, 
    and its reciprocal roots are precisely the roots of the Hasse--Weil zeta function of $X_n$ over $k_n$.
    It follows from \cite[\S1.3]{Katz} (see also \cite[\S14.4--\S14.5]{Kedlaya}) that the Newton polygon of the $F$-isocrystal $D_n$
    coincides with the vertically scaled $p$-adic Newton polygon $h_n^{-1}\NP_p(\mathcal{C}(X_n/k_n,s))$. 
    In particular, these {\em normalized} $p$-adic Newton polygons 
    are independent of the choice of finite field $k_n$ over which $X_n$ is defined, and our work 
    provides an Iwasawa-theoretic description of them as $n\rightarrow\infty$.  In this way, 
    we establish the {\em unramified} analogue of Wan's program to understand $p$-adic Newton
    slopes of Hasse--Weil $L$-functions in {\em ramified} $\Z_p$-towers of curves over a (fixed) finite field (\cite{WanSlopes2}, \cite{Xiang}, \cite{KostersZhu}, \cite{KMU1,KMU2}; see also \cite{rwxy} and \cite{RenSlopes}).
    Note that, by class field theory, a global function field in one variable over a {\em finite}
    field admits no everywhere unramified geometric $\Z_p$-extensions whatsoever, so working 
    over infinite $k$ is essential in our context. 
    An important feature of our work is that it applies over {\em any} algebraically closed $k$
    of characteristic $p$, so it permits interesting questions about 
    the Iwasawa-theoretic behavior of slopes in the context of (unramified) $\Z_p$-extensions of
    {\em generic} curves of genus $g$ and $p$-rank $f$ (although we do not pursue this here).
\end{rem}

The starting point for our work is \cite[Theorem B (i)]{ID} (Theorem \ref{IDmain} below), 
which provides a finite, free, and self-dual Frobenius module $M$ over the Iwasawa algebra 
$\Lambda\coloneq W[\![\Z_p]\!]\simeq W[\![T]\!]$
interpolating all of the isocrystals $D_n^{(0,1)}$.  
It follows formally from this that the
quotient $F$-isocrystal $D_{n}^{(0,1)}/D_{n-1}^{(0,1)}$
decomposes, after extending scalars to $K_n$, as the direct sum over
all primitive $p^n$-th roots of unity $\zeta$ of the $F$-isocrystals $M_{\zeta}\coloneq K_n\otimes_{\Lambda} M$
arising from $M$ by specializing $T\mapsto \zeta-1$; see Corollary \ref{cor:directsum}.
Crucially using the hypothesis that $k$ is algebraically closed, 
we prove that every $M_{\zeta}$ admits a {\em cyclic vector}
(see \S\ref{sec:cyclic}), and deduce from this (via the Baire category theorem) that $M$ itself admits a cyclic
vector after an appropriate localization; see Theorem \ref{cyclic}.  
Theorem \ref{thm:main-intro} and Corollary \ref{cor:main} are then straightforward consequences
of Kedlaya's theory \cite{Kedlaya} of {\em difference modules} over general {\em difference rings},
and a variant of the Grothendieck--Katz semicontinuity theorem (see Theorem \ref{thm:NPabstract}).

\subsection*{Acknowledgements}

We are grateful to Kiran Kedlaya for several helpful discussions and suggestions.
This project originated with the 2024 Ph.D.~thesis of the second author \cite{Lewis}, which established a preliminary version of Theorem \ref{thm:main-intro} in the case $d=1$, and for general $d$ under additional hypotheses (established by Theorem \ref{cyclic} below).

\subsection*{Tool and computational resource disclosure} 

Preliminary versions of the main results of this paper were discovered and proved without
computer assistance.  Discussions with ChatGPT helped us simplify several proofs and strengthen
the statements of Theorem \ref{thm:main-intro} and Corollary \ref{cor:main}.  We used ChatGPT to proofread the manuscript 
and suggest improvements to the exposition.

\section{Iwasawa--Dieudonn\'e modules}

We continue with the setup and notation of \S\ref{intro}, and throughout this and the following section
use the terminology of \cite{Kedlaya}.
Let $\Gscr_n\coloneq J_{X_n}[p^{\infty}]$ be
the $p$-divisible group of the Jacobian of $X_n$ over $k$.  
By \cite{MM}, writing $\D(\cdot)$ for the contravariant
Dieudonn\'e module functor, 
we have a canonical isomorphism
$\D(\Gscr_n)\simeq H^1_{\crys}(X_n/W)$ of $F$-crystals over 
$W$. 

\begin{defn}
    We write $M_n\coloneq \D(\Gscr_n^{\ll})$ for the (contravariant) Dieudonn\'e module 
    of the local--local component $\Gscr_n^{\ll}$ of $\Gscr_n$; 
    we have $M_n[1/p]\simeq D_n^{(0,1)}$ as $F$-isocrystals over $K_0$.
\end{defn}

For $n\ge 1$, pullback of line bundles along the natural maps $\pi:X_n\rightarrow X_{n-1}$
induces maps of Jacobians, and hence of $p$-divisible groups $\pi^*:\Gscr_{n-1}\rightarrow \Gscr_n$
over $k$.  Passing to local--local parts and applying the contravariant Dieudonn\'e module functor
yields maps of Dieudonn\'e modules
$\tau: M_n\rightarrow M_{n-1} $
and thus a projective system $\{M_n,\tau\}$.  This projective system admits an action of the Iwasawa algebra 
$\Lambda\coloneq W[\![\Z_p]\!] \simeq W[\![T]\!]$, wherein we make the usual
identification of $W$-algebras taking a fixed choice of topological generator of $\Z_p$ to $1+T$.
We write $\varphi$ for the unique automorphism of $W$ lifting the $p$-power map on $k$,
and continuously extend $\varphi$ to $\Lambda$ by $\varphi(T)=T$; this
makes $\Lambda$ into an {\em inversive difference ring} in the sense of \cite[\S14]{Kedlaya}.
For $n\ge 0$ we define 
$$\Lambda_n\coloneq W[\Z/p^n\Z] \simeq W[\![T]\!]/(\rho_n), \quad\text{where}\quad \rho_n=(1+T)^{p^n}-1.$$

\begin{defn}
    The {\em local--local  Iwasawa--Dieudonn\'e module} associated to $\{X_n\}_{n\ge 0}$ is 
    $$
        M\coloneq \varprojlim_{n,\tau} M_n.
    $$
    It is naturally a $\Lambda$-module with additive maps $F$ and $V$ that are $\Lambda$-semilinear 
    over $\varphi$ and $\varphi^{-1},$ respectively, and satisfy $FV=VF=p$.
    In this way, the pair $(M,F)$ is a difference module over $\Lambda$.
\end{defn}

\begin{thm}\label{IDmain}
    The $\Lambda$-module $M$ is free of rank $2d$ for $d\coloneq g-f$.  For $n\ge 0$, projection induces an isomorphism of 
    $\Lambda_n$-modules $\Lambda_n\otimes_{\Lambda} M \simeq M_n$ with action of $F$ and $V$.  
\end{thm}

\begin{proof}
    This is \cite[Theorem B (i)]{ID}.
\end{proof}

As in Remark \ref{rem:delta}, we put $K_n\coloneq K_0(\mu_{p^n})$, and
we extend the automorphism $\varphi$ of $K_0$ to $K_n$ by acting trivially on $\mu_{p^n}$,
which makes $K_n$ into a {\em difference field}.
For any $p^n$-th root of unity $\zeta$ in $K_n$, we write $\varepsilon_{\zeta} : \Lambda \rightarrow K_n$
for the $W$-algebra map taking $T$ to $\zeta-1$; this is a map of difference rings, so the
scalar extension $M_{\zeta}\coloneq K_n\otimes_{\Lambda,\varepsilon_{\zeta}} M$ is naturally a difference
module over $K_n$.
Albanese functoriality of the Jacobian induces maps of $p$-divisible groups $\pi_*: \Gscr_{n}\rightarrow \Gscr_{n-1}$,
and hence of local--local Dieudonn\'e modules $\sigma: M_{n-1}\rightarrow M_n$.  
Setting $\rho_{-1}\coloneq 1$, for $n\ge 0$ we write $\Phi_n\coloneq \rho_n/\rho_{n-1}\in \Z[T]$ for the $p^n$-th cyclotomic polynomial evaluated at $1+T$.

\begin{cor}\label{cor:directsum}
        For $n\ge 1$, the map $\sigma: M_{n-1}\rightarrow M_n$ is an inclusion of Dieudonn\'e modules.
        It induces a canonical direct sum decomposition of difference modules over $K_n$
        \begin{equation}\label{eq:directsum}
            K_n\otimes_{K_0} D_n^{(0,1)}  \simeq \left(K_n\otimes_{K_0}D_{n-1}^{(0,1)}\right)\oplus \bigoplus_{\zeta: \Phi_n(\zeta-1)=0} 
            M_{\zeta}.
    \end{equation}
\end{cor}

\begin{proof}
        The composition $\pi_*\pi^*$
is multiplication by $p$ on $\Gscr_{n-1}$, so the map $\sigma$ is an inclusion of Dieudonn\'e modules 
$M_{n-1}\hookrightarrow M_n$, and the map $p^{-1} \sigma[1/p] : M_{n-1}[1/p] \rightarrow M_{n}[1/p]$
is a section of the projection $\tau[1/p]: M_n[1/p] \rightarrow M_{n-1}[1/p]$.  This yields an isomorphism of
$F$-isocrystals $D_n^{(0,1)} \simeq D_{n-1}^{(0,1)} \oplus \ker (\tau[1/p])$ over $K_0$.
On the other hand, Theorem \ref{IDmain} gives identifications
$$
    \ker(\tau) \simeq \rho_{n-1} M / \rho_n M = \rho_{n-1} M / \rho_{n-1} \Phi_n M  \simeq M/\Phi_nM 
$$
wherein the final isomorphism is induced by multiplication by $\rho_{n-1}$ on $M$. 
As $\rho_{n-1}$ is fixed by $\varphi$, these are identifications of Dieudonn\'e modules.  Inverting $p$
and using the fact that the formation of kernels commutes with localization yields a canonical
identification of $F$-isocrystals over $K_0$
\begin{equation}
                D_n^{(0,1)} \simeq D_{n-1}^{(0,1)} \oplus (M/\Phi_n M)[1/p].\label{eq:isocdecomp}
\end{equation}
Extending scalars to $K_n$, the identification
\eqref{eq:directsum} follows from the Chinese Remainder Theorem.
\end{proof}

The automorphism $\varphi$ of $\Lambda$ extends uniquely and continuously to $\E$, making it an inversive difference field satisfying \cite[Hypothesis 14.4.1]{Kedlaya}.
The scalar extension $M_{\E}\coloneq \E\otimes_{\Lambda}M$
is then a finite dualizable difference module over $\E$ \cite[Definition 14.1.4]{Kedlaya}
of $\E$-dimension $2d$ for $d=g-f$,
and we write $\NP(M_{\E})$ for its Newton polygon 
\cite[Definition 14.5.2]{Kedlaya}.  Likewise, for each primitive $p^n$-th root of unity $\zeta$, the scalar extension 
$M_{\zeta}\coloneq K_n\otimes_{\Lambda,\varepsilon_{\zeta}}M$
is a finite dualizable difference module over the difference field $K_n$, and we write $\NP(M_{\zeta})$
for its Newton polygon.

\begin{thm}\label{thm:NPabstract}
    The Newton polygon $\NP(M_{\zeta})$
    is independent of the choice of primitive $p^n$-th root of unity $\zeta$.
    It lies on or above $\NP(M_{\E})$, with the same endpoints, and converges
    uniformly to $\NP(M_{\E})$ as $n\rightarrow \infty$.
    The ``generic'' Newton polygon $\NP(M_{\E})$ has endpoints $(0,0)$ and $(2d,d)$, with all slopes in $[0,1]$ and 
    underlying multiset of slopes stable under $s\mapsto 1-s$.
\end{thm}

\begin{proof}
    As $\Gal(K_n/K_0)$ acts isometrically on $K_n$, commutes with $\varphi$, and transitively permutes the roots of $\Phi_n=0$,
    it follows that 
    $\NP(M_{\zeta})$ is independent of $\zeta$.  Each $F$-isocrystal $D_n^{(0,1)}$
    has slopes in $(0,1)$ with underlying multiset of slopes stable under $s\mapsto 1-s$ by autoduality
    of the Jacobian, so Corollary \ref{cor:directsum} implies that the same
    is true of $\NP(M_{\zeta})$ for all $n$.  Since $M_{\zeta}$ is of dimension $2d$ over $K_n$ and $(0,0)$ is the initial
    point of $\NP(M_{\zeta})$, the terminal point is $(2d,d)$ by symmetry.
    
    To see that $\NP(M_{\zeta})$ lies on or above $\NP(M_{\E})$, and converges uniformly to it as $n\rightarrow \infty$,
    we adapt the proof of the Grothendieck--Katz semicontinuity theorem \cite[Theorem 15.3.2]{Kedlaya}.
    Let $v_{\E}$ denote the Gauss valuation on $\Lambda$, given by $v_{\E}(\sum a_j T^j)\coloneq \min_j \ord_p(a_j)$.
  For $n\ge 1$, let $v_{n}\coloneq \ord_p\circ \varepsilon_{\zeta}$
    with $\zeta$ any primitive $p^n$-th root of unity.  
    It follows easily from Weierstrass preparation in $\Lambda$ that
    \begin{equation}\label{eq:valbound}
        v_{n}(a) \ge v_{\E}(a)\ \text{for all}\ a\in \Lambda,\text{and}\lim_{n\rightarrow \infty} v_{n}(a) = v_{\E}(a)\ \text{for each fixed}\ a\in \Lambda.
    \end{equation}
    Denote by $H_{m,n}$ (respectively $H_{m,\E}$) the Hodge polygon of $F^m$ acting on $M$, computed using $v_{n}$ (respectively $v_{\E}$) \cite[Definition 14.5.1]{Kedlaya}.
    For each integer $i \in [1,2d]$, it follows from \eqref{eq:valbound} and the definition of the ordinate $H_{m,\star}(i)$
    as the minimum of $v_{\star}(\Delta)$ over all $i\times i$ minors $\Delta$ of the matrix of $F^m$ on $M$ (with respect to
    any fixed $\Lambda$-basis) \cite[Definition 4.3.3]{Kedlaya}
    that
    \begin{equation}\label{eq:hodgeconv}
            H_{m,n}(i) \ge H_{m,\E}(i)\ \text{for all}\ n, \text{and}\ \lim_{n\rightarrow\infty}H_{m,n}(i) =  H_{m,\E}(i).
    \end{equation}
    Now consider the difference module $\wedge^i(\O_{K_n}\otimes_{\Lambda,\varepsilon_{\zeta}}M)$, which has rank $r_i\coloneq \binom{2d}{i}$, least Newton slope $\NP(M_{\zeta})(i)\le i$, and least Hodge slope for $F^m$ equal to $H_{m,n}(i)$.
    Katz's {\em basic slope estimate}\footnote{While \cite[1.4.3]{Katz} is proved for $F$-crystals over $W$, 
    the proof works {\em mutatis mutandis} in our slightly more general context, using
\cite[Corollary 14.6.4]{Kedlaya} for the required Dieudonn\'e--Manin decomposition, and \cite[Corollary 14.5.4]{Kedlaya}
for the ``Newton above Hodge'' inequality.}
    \cite[1.4.3]{Katz} yields:
    \begin{equation}\label{eq:Katzslope}
            0 \le \NP(M_{\zeta})(i) - \frac{1}{m}H_{m,n}(i) \le \frac{i(r_i-1)}{m}
    \end{equation}
    for all $m \ge r_i$.  The key point is that the right hand side is {\em independent} of $n$.
    On the other hand, for each $i$ with $0\le i\le2d$, \cite[Proposition 14.5.8]{Kedlaya} gives
    \begin{equation}\label{eq:Kedlayaslope}
        \lim_{m\rightarrow \infty} \frac{1}{m} H_{m,\E}(i) = \NP(M_{\E})(i).
    \end{equation}
    For each $i$ and $\epsilon>0$, we may therefore choose $m$ so that 
    $|\NP(M_{\zeta})(i)-\frac{1}{m}H_{m,n}(i)|< \epsilon/3$ for all $n$ and $|\NP(M_{\E})(i)-\frac{1}{m}H_{m,\E}(i)|< \epsilon/3$.
    It follows from \eqref{eq:hodgeconv} that for all $n$ sufficiently large we have
    $|\frac{1}{m}H_{m,n}(i)-\frac{1}{m}H_{m,\E}(i)| < \epsilon/3$, and the triangle inequality then yields
    $$
        |\NP(M_{\zeta})(i)-\NP(M_{\E})(i)| < \epsilon.
    $$
    As we need only consider the finitely many integers $i\in [0,2d]$ and $\NP(M_{\star})$ is linear
    on each interval $[i,i+1]$, we conclude that $\NP(M_{\zeta})$ converges uniformly to $\NP(M_{\E})$ as $n\rightarrow \infty$.
    Moreover, dividing the Hodge polygon bound in \eqref{eq:hodgeconv} by $m$, taking the limit as $m\rightarrow \infty$
    and using \eqref{eq:Katzslope} and \eqref{eq:Kedlayaslope}, we deduce that $\NP(M_{\zeta})$ lies on or above $\NP(M_{\E})$ for all $\zeta$.
    Finally, since $\NP(M_{\zeta})$ has all slopes in $(0,1)$ with underlying slope multiset stable under $s\rightarrow 1-s$ and converges uniformly to $\NP(M_{\E})$, we conclude that $\NP(M_{\E})$ has all slopes in $[0,1]$
    and underlying slope multiset stable under $s\mapsto 1-s$; as its initial point is $(0,0)$,
    it has terminal point $(2d,d)$ by symmetry.
\end{proof}

\begin{rem}
    It follows from \cite[Theorem B (i)]{ID} that one has a natural isomorphism of difference modules $M_{\E}\simeq \E(1)\otimes_{\E} M_{\E}^{\vee}$
    where $M_{\E}^{\vee}$ is the
    {\em dual} in the sense of \cite[Definition 2.6]{ID} and $\E(1)\coloneq \E\cdot v$ is the rank-1
    difference module over $\E$ with Frobenius determined by $Fv=pv$.  This yields another proof that the underlying
    slope multiset of $\NP(M_{\E})$ is stable under $s\mapsto 1-s$.
\end{rem}

\section{Cyclic vectors}\label{sec:cyclic}

In this section, we will establish an explicit description of the Newton polygon 
  of the specialization $M_{\zeta}$ in terms of a ``universal'' polynomial $P$ with coefficients
  in an appropriate localization of $\Lambda$. Together with 
  Theorem \ref{thm:NPabstract}, this will complete the proof of
Theorem \ref{thm:main-intro}.
If $d=0$, Theorem \ref{thm:main-intro} and Corollary \ref{cor:main}
hold with $P=\delta=1$, so throughout this section we assume $d>0$.

Recall \cite[Definition 14.2.1]{Kedlaya} that for a difference ring
$(R,\varphi)$, the {\em twisted polynomial ring} $R\{F\}$ is the ring of finite
formal sums $\sum_{i\ge0}r_iF^i$, with multiplication determined by
$Fr=\varphi(r)F$. 
Given $P\in R\{F\}$, the set $R\{F\}P$ is a left ideal of $R\{F\}$, and the $R$-module $R\{F\}/R\{F\}P$
is a difference module over $R$ with Frobenius given by left multiplication by $F$.
For a finite free difference module $M$ of rank $r$ over $R$,
we call $m\in M$ a \emph{cyclic vector} if
$m,Fm,\ldots,F^{r-1}m$ form an $R$-basis of $M$.
Equivalently, the map $R\{F\}\to M$ sending $1$ to $m$ induces
an isomorphism of difference modules $R\{F\}/R\{F\}P\simeq M$ for a monic $P$ of
degree $r$. This agrees with Kedlaya's definition when $R$ is
an inversive integral domain. More generally, for an $R$-algebra
$A$, we say that $m$ is a \emph{cyclic vector over $A$} if
$1\otimes m,\ldots,1\otimes F^{r-1}m$ form an $A$-basis of
$A\otimes_R M$.

\begin{lem}\label{kedlem}
        Let $(K,\varphi)$ be a difference field.  Assume there exists $\alpha\in K$ satisfying:
        \begin{enumerate}
            \item $\varphi(\alpha)=\alpha$
            \item $\alpha$ is not a root of unity
            \item $\varphi(x)=\alpha x$ admits a nonzero solution $x\in K$. 
        \end{enumerate}
        Then every finite dualizable difference module
        over $K$ admits a cyclic vector.
\end{lem}

\begin{proof}
    Adapt the proof of \cite[Theorem 5.4.2]{Kedlaya} (see \cite[Chapter 14, Exercise (3)]{Kedlaya}).
\end{proof}

\begin{cor}\label{cor:cyclic}
    Any finite dualizable difference module over $K_n$ admits a cyclic vector.
\end{cor}

\begin{proof}
    By Lemma \ref{kedlem}, it suffices to exhibit $\alpha \in K_n$
    satisfying the three conditions.  This is straightforward: choose  
    $\alpha\in \Z_p^{\times}$ of infinite order
    (e.g.~$\alpha=1+p$), and note that $x\mapsto \varphi(x)/x$ is surjective on $W^{\times}$
    as $k$ is algebraically closed.  Alternatively, the $F$-isocrystal $K_0\cdot v$ with $Fv=\alpha^{-1}\cdot v$
    has slope zero, so admits an $F$-fixed generator $x\cdot v$ by the Dieudonn\'e--Manin classification.
\end{proof}

\begin{thm}\label{cyclic}
    There exist $m\in M$ and $h\in\Lambda$, with
    $h\notin\Phi_n\Lambda$ for all $n\ge0$, such that
    $m$ is a cyclic vector over $\Lambda[1/h]$.
\end{thm}

\begin{proof}
    Fix a $\Lambda$-basis of $M$, and for $m\in M$ define
    $$
        \Delta(m)\coloneq \det(m,Fm,\ldots,F^{2d-1}m)\in \Lambda
    $$
    computed with respect to the fixed choice of basis.
    For $n\ge 0$, set 
    $$
        U_n\coloneq \left\{m \in M\ :\ \Delta(m)\not\in \Phi_n\Lambda\right\}.
    $$
    Equipping $M$ with the $p$-adic topology, for which it is a complete metric space, we claim that $U_n$ is an open and dense
    subset of $M$.  First observe that $(M/\Phi_nM)[1/p]$ is a finite, free, dualizable difference module over $(\Lambda/\Phi_n\Lambda)[1/p]\simeq K_n$.
    By Corollary \ref{cor:cyclic}, it admits a cyclic vector. Multiplying such a cyclic vector by a sufficiently high power of $p$,
    we may thereby choose $y\in M$ whose image in $M/\Phi_nM$ is a cyclic vector after inverting $p$, so $\Delta(y)\not\in \Phi_n\Lambda$
    and $y\in U_n$.  For $N\ge 0$, $w\in M$, and $t\in \Z_p$, we compute that
    $\Delta(w+p^Nty)$ is a polynomial in $\Lambda[t]$ of degree $2d$ with leading term
    $
        p^{2dN}t^{2d}\Delta(y)\not\equiv 0 \bmod \Phi_n.
    $
    In particular, this polynomial is nonzero and 
its reduction modulo $\Phi_n$ has only finitely many roots in the integral domain $\Lambda/\Phi_n\Lambda$.    
    As $\Z_p$ is infinite, there exists $t\in \Z_p$ so that $\Delta(w+p^Nty)\not\equiv 0 \bmod \Phi_n$
    and we conclude that every $p$-adic neighborhood of $w$ meets $U_n$.  Thus, $U_n$ is dense in $M$.
    Since $\Lambda/\Phi_n\Lambda$ is finite and free over $W$, it is $p$-adically Hausdorff, so the complement
    of zero is open; as the map $m\mapsto\Delta(m)\bmod\Phi_n$ is $p$-adically continuous, $U_n$ is open.
  By the Baire category theorem applied to the complete metric
    space $M$, the intersection $\bigcap_{n\ge 0} U_n$ is {\em nonempty}.  We may therefore choose $m$ in this intersection,
    and we set $h\coloneq \Delta(m)$.  By construction, $h\not\in \Phi_n\Lambda$ for all $n\ge 0$, and 
    $m$ is a cyclic vector over $\Lambda[1/h]$.
\end{proof}

\begin{cor}\label{cor:maincyclic}
    There exist $\delta\in W[T]$ with $\delta(\zeta-1)\neq 0$ for every $p$-power root of unity $\zeta\in \o{K}_0$,
    and a monic polynomial 
    $P$ of degree $2d$ with coefficients in $\Lambda[1/\delta]\cap \O_{\E} \subseteq \E$
    such that 
    $$
        \NP(M_{\zeta}) = \NP_p(P\big|_{T=\zeta-1})
    $$
    for all $n\ge 1$ and each primitive $p^n$-th root of unity $\zeta$.
\end{cor}

\begin{proof}
By Theorem \ref{cyclic}, we may choose $m\in M$ and $h\in\Lambda$ with
$h\notin\Phi_n\Lambda$ for all $n\ge0$ such that $m$ is a cyclic vector
over $\Lambda[1/h]$. The resulting monic relation among
$m,Fm,\ldots,F^{2d}m$ defines a polynomial $P$ of degree $2d$
with coefficients in $\Lambda[1/h]$.
By Weierstrass preparation, we may write $h=\delta u$ with
$\delta\in W[T]$ and $u\in\Lambda^\times$.
Then $P$ has coefficients in $\Lambda[1/\delta]$, and
$\delta(\zeta-1)\ne0$ for every $p$-power root of unity $\zeta$,
since $h\notin\Phi_n\Lambda$ for all $n\ge0$.
For each primitive $p^n$-th root of unity $\zeta$, the map
$\varepsilon_\zeta:\Lambda\to K_n$ factors through $\Lambda[1/\delta]$, and we
obtain an isomorphism of difference modules 
\[
M_\zeta\simeq
K_n\{F\}/K_n\{F\}(P\big|_{T=\zeta-1})
\]
over $K_n$.  Likewise, extending scalars to $\E$ gives an isomorphism of difference modules over $\E$
\[
M_\E\simeq \E\{F\}/\E\{F\}P.
\]
By \cite[Proposition 14.5.7]{Kedlaya}, these isomorphisms yield
\[
\NP(M_\zeta)=\NP_p(P\big|_{T=\zeta-1}),
\qquad
\NP(M_\E)=\NP_\E(P).
\]
Finally, Theorem \ref{thm:NPabstract} shows that $\NP_\E(P)$
lies on or above the horizontal axis. By the definition of the
Newton polygon as a lower convex hull, the valuation of each
nonzero coefficient is at least the ordinate of the polygon
at the corresponding abscissa, and is therefore nonnegative.
Thus $P$ has coefficients in $\Lambda[1/\delta]\cap\O_\E$.
\end{proof}

\begin{proof}[Proof of Theorem \ref{thm:main-intro} and Corollary \ref{cor:main}]
  Combining Corollary \ref{cor:directsum}, Theorem \ref{thm:NPabstract}, and Corollary \ref{cor:maincyclic} gives Theorem \ref{thm:main-intro}.  To prove Corollary \ref{cor:main}, write $P=\sum_{0\le i \le 2d} a_{i}F^i$ with $a_i\in \Lambda[1/\delta]\cap \O_{\E}$ and $a_{2d}\coloneq 1$. 
  By Weierstrass preparation, for each $i\in I\coloneq \{i\ :\ a_i\neq 0\},$ we may write
 $$
    a_i = p^{\alpha_i} \frac{H_i}{G_i}u_i
 $$
 with $\alpha_i \in \Z_{\ge 0}$, $u_i\in \Lambda^{\times}$, and $H_i,G_i \in W[T]$ monic distinguished polynomials.  
 Writing $\Delta_i\coloneq \deg(H_i)-\deg(G_i)$ and $e_n\coloneq p^{n-1}(p-1)$,
 it follows that for all $n\gg0$, the $p$-adic Newton polygon of $P\big|_{T=\zeta-1}$
 is the lower convex hull of the points $q_i(n)\coloneq (2d-i,\alpha_{i}+ e_n^{-1}\Delta_{i})$ for $i\in I$.  
 The indices of the vertices of this lower convex hull are independent of $n$ for $n\gg0$: 
 indeed, for $i,j,k\in I$ with $j < i < k$, the signed vertical displacement of the point $q_{i}(n)$
 from the line through $q_{j}(n)$ and $q_{k}(n)$ is visibly an {\em affine} function of $e_n^{-1}$ with rational
 coefficients, so each of the finitely many such displacements is either identically zero or has a fixed sign for all sufficiently large $n$. We conclude that for all $n\gg0$, the $i$-th ordered slope of $\NP_p(P\big|_{T=\zeta-1})$ is of the form
 $\lambda_i+e_n^{-1}b_i$ with $\lambda_i,b_i\in \Q$ for $i=1,\ldots,2d$.  
 Since $\NP_p(P\big|_{T=\zeta-1})$ converges uniformly to $\NP_{\E}(P)$
 as $n\rightarrow \infty$, we must have $\lambda_i=s_i$, the $i$-th ordered slope of the generic Newton polygon $\NP_{\E}(P)$.
 As both the generic and specialized Newton polygons have underlying multiset of slopes that is stable under 
 $s\mapsto 1-s$, we have $s_{2d+1-i}=1-s_i$ and $b_{2d+1-i}=-b_i$ for $i=1,\ldots,d$.  Finally, the fact that  
 $\NP_p(P\big|_{T=\zeta-1})$
 lies on or above $\NP_{\E}(P)$ with the same endpoints yields $\sum_{i=1}^j b_i \ge 0$ for $j=1,\ldots,d$
 and our fixed ordering of the specialized slopes forces
 $b_{i+1}\ge b_i$ whenever $s_{i+1}=s_i$ for some $i$ with $1\le i < d$.
\end{proof}

\bibliographystyle{amsalpha}
\bibliography{iw}

\end{document}